\documentclass[11pt]{amsart}
\usepackage[margin=1in]{geometry}

\usepackage{amssymb}
\usepackage{amsthm}
\usepackage{amsmath}
\usepackage{mathrsfs}
\usepackage{amsbsy}
\usepackage[all]{xy}
\usepackage{bm}
\usepackage{hyperref}
\usepackage{tikz}
\usepackage{array}
\usepackage{float}
\usepackage{enumerate}
\usepackage{xcolor}

\usepackage{hhline}
\allowdisplaybreaks
\usepackage[noadjust]{cite}
\usepackage{asymptote}

\usepackage{caption}
\usepackage{tabu}
\usepackage{diagbox}

\newenvironment{enumerate*}%
  {\begin{enumerate}[(i)]%
    \setlength{\itemsep}{10pt}%
    \setlength{\parskip}{0pt}}%
  {\end{enumerate}}

\newtheorem{theorem}{Theorem}[section]
\newtheorem{proposition}[theorem]{Proposition}
\newtheorem{corollary}[theorem]{Corollary}

\newtheorem{lemma}[theorem]{Lemma}

\theoremstyle{definition}

\newtheorem{remark}[theorem]{Remark}

\DeclareMathOperator{\sgn}{sgn}
\DeclareMathOperator{\FS}{\mathsf{FS}}
\DeclareMathOperator{\Star}{\mathsf{Star}}
\DeclareMathOperator{\Path}{\mathsf{Path}}
\DeclareMathOperator{\Cycle}{\mathsf{Cycle}}
\DeclareMathOperator{\Spider}{\mathsf{Spider}}
\DeclareMathOperator{\Dand}{\mathsf{Dand}}

\newcommand{\Fruit}{\mathsf{Cycle}^{\perp}}

\newcommand{\dfn}[1]{\textcolor{blue}{\emph{#1}}}

\begin{document}

\title[]{Connectedness and Cycle Spaces of Friends-and-Strangers Graphs}
\subjclass[2010]{}

\author[Colin Defant]{Colin Defant}
\address[]{Department of Mathematics, Massachusetts Institute of Technology, Cambridge, MA 02139, USA}
\email{colindefant@gmail.com}

\author[David Dong]{David Dong}
\address[]{Eastside Preparatory School, Kirkland, WA 98004, USA}
\email{etsptq@gmail.com}

\author[Alan Lee]{Alan Lee}
\address[]{Henry M. Gunn High School, Palo Alto, CA 94306, USA}
\email{alandongjinlee@gmail.com}

\author[Michelle Wei]{Michelle Wei}
\address[]{The Harker School, San Jose, CA 95124, USA}
\email{michelle.wei89@gmail.com}

\maketitle

\begin{abstract}
If $X=(V(X),E(X))$ and $Y=(V(Y),E(Y))$ are $n$-vertex graphs, then their \emph{friends-and-strangers graph} $\mathsf{FS}(X,Y)$ is the graph whose vertices are the bijections from $V(X)$ to $V(Y)$ in which two bijections $\sigma$ and $\sigma'$ are adjacent if and only if there is an edge $\{a,b\}\in E(X)$ such that $\{\sigma(a),\sigma(b)\}\in E(Y)$ and $\sigma'=\sigma\circ (a\,\,b)$, where $(a\,\,b)$ is the permutation of $V(X)$ that swaps $a$ and $b$. We prove general theorems that provide necessary and/or sufficient conditions for $\mathsf{FS}(X,Y)$ to be connected. As a corollary, we obtain a complete characterization of the graphs $Y$ such that $\mathsf{FS}(\mathsf{Dand}_{k,n},Y)$ is connected, where $\mathsf{Dand}_{k,n}$ is a dandelion graph; this substantially generalizes a theorem of the first author and Kravitz in the case $k=3$. For specific choices of $Y$, we characterize the spider graphs $X$ such that $\mathsf{FS}(X,Y)$ is connected. In a different vein, we study the cycle spaces of friends-and-strangers graphs. Naatz proved that if $X$ is a path graph, then the cycle space of $\mathsf{FS}(X,Y)$ is spanned by $4$-cycles and $6$-cycles; we show that the same statement holds when $X$ is a cycle and $Y$ has domination number at least $3$. When $X$ is a cycle and $Y$ has domination number at least $2$, our proof sheds light on how walks in $\mathsf{FS}(X,Y)$ behave under certain \emph{Coxeter moves}. 
\end{abstract}

\tableofcontents

\section{Introduction}\label{SecIntro}

\emph{Flip graphs} are graphs that encode when combinatorial objects are related by small changes called \emph{flips}. In recent years, these graphs have received a great deal of attention in combinatorics, geometry, and computer science. For example, some of the most well-studied flip graphs are the $1$-skeleta of polytopes such as the permutahedron, the associahedron, and the cyclohedron. In \cite{friends}, the first author and Kravitz introduced a broad family of flip graphs called \emph{friends-and-strangers graphs}. 

Suppose we are given simple graphs $X=(V(X),E(X))$ and $Y=(V(Y),E(Y))$ such that $|V(X)|=|V(Y)|=n$. We imagine that the vertices of $X$ are chairs and that the vertices of $Y$ are people; two people are adjacent in $Y$ if and only if they are friends with each other (otherwise, they are strangers). The \dfn{friends-and-strangers graph} of $X$ and $Y$, denoted $\FS(X,Y)$, is a graph whose vertices are the bijections from $V(X)$ to $V(Y)$; one can think of such a bijection as an arrangement of people sitting in chairs. If we are given such an arrangement, then we allow two people to swap places with each other if they are friends with each other and they are sitting in adjacent chairs. Such a swap is called an \dfn{$(X,Y)$-friendly swap}. The edges of $\FS(X,Y)$ correspond precisely to $(X,Y)$-friendly swaps. More formally, two bijections $\sigma,\sigma':V(X)\to V(Y)$ are adjacent in $\FS(X,Y)$ if and only if there exists an edge $\{a,b\}\in E(X)$ such that $\{\sigma(a),\sigma(b)\}\in E(Y)$, $\sigma(a)=\sigma'(b)$, $\sigma(b)=\sigma'(a)$, and $\sigma(c)=\sigma'(c)$ for all $c\in V(X)\setminus\{a,b\}$. For example, suppose \[X = Y = \begin{array}{l}\includegraphics[height=.5cm]{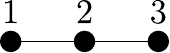}.
\end{array}\]


 Since $V(X)=V(Y)=\{1,2,3\}$, we can represent each bijection $\sigma\colon V(X)\to V(Y)$ as a permutation $\sigma(1)\sigma(2)\sigma(3)$ in one-line notation. Then
\[\FS(X,Y)= \begin{array}{l}\includegraphics[height=1.255cm]{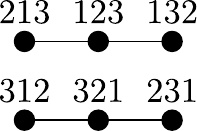}
\end{array}.\]

Friends-and-strangers graphs generalize several previously-studied notions. For example, when $Y=K_n$ is the complete graph with vertex set $[n]:=\{1,\ldots,n\}$ and $X$ is another graph with vertex set $[n]$, the friends-and-strangers graph $\FS(X,Y)$ is the Cayley graph of the symmetric group $\mathfrak{S}_n$ generated by the collection of transpositions corresponding to the edges of $X$. The famous $15$-puzzle is equivalent to analyzing $\FS(\Star_{16},\mathsf{Grid}_{4\times 4})$, where $\Star_{n}$ is the star graph with $n$ vertices and $\mathsf{Grid}_{4\times 4}$ is the $4\times 4$ grid graph. Generalizing the $15$-puzzle, Wilson \cite{wilson} characterized the graphs $Y$ such that $\FS(\Star_{n},Y)$ is connected (see Theorem~\ref{thm:wilson}). Stanley \cite{stanley} studied the connected components of $\FS(\Path_n,\Path_n)$, where $\Path_n$ is the path graph with $n$ vertices.  

It is very natural to ask about the connected components of a friends-and-strangers graph; indeed, if we view the vertices of $\FS(X,Y)$ as arrangements of people in chairs, then two such arrangements are in the same connected component of $\FS(X,Y)$ if and only if one can be obtained from the other via a sequence of $(X,Y)$-friendly swaps. As mentioned above, Wilson studied this problem when $X$ is a star graph. Several papers have continued this line of work when one of $X$ or $Y$ is fixed to be a specific type of graph \cite{friends, Jeong}, when $X$ and $Y$ are Erd\H{o}s--R\'enyi random graphs \cite{AlonDefantKravitz,Wang}, or when $X$ and $Y$ satisfy certain minimum-degree conditions \cite{AlonDefantKravitz, Bangachev}. Jeong recently studied the girths and diameters of friends-and-strangers graphs \cite{JeongStructural}. The first author has also related friends-and-strangers graphs of the form $\FS(\Cycle_n,Y)$, where $\Cycle_n$ is the cycle with $n$ vertices, to a dynamical system called \emph{toric promotion} \cite{ToricPromotion}. 

In \cite{friends}, the first author and Kravitz found sufficient conditions for $\FS(X,Y)$ to be connected under the hypothesis that $X$ has a Hamiltonian path. One of our primary goals in this paper is to prove the following sufficient condition for $\FS(X,Y)$ to be connected; this condition is very general and is quite different from the ones established in \cite{friends}.

\begin{theorem}\label{thm:spider_sufficient}
Let $X$ and $Y$ be connected $n$-vertex graphs such that $X$ has maximum degree $k\geq 2$. If every $k$-vertex induced subgraph of $Y$ is connected and there exists a $k$-vertex induced subgraph $Y_0$ of $Y$ such that $\FS(\Star_{k},Y_0)$ is connected,
then $\FS(X,Y)$ is connected.
\end{theorem}

In practice, Wilson's theorem, which we have recorded as Theorem~\ref{thm:wilson} below, makes it easy to check whether or not $\FS(\Star_k,Y_0)$ is connected for a given graph $Y_0$. In fact, Wilson's result will allow us to prove the following, which immediately lets us ignore one of the hypotheses of Theorem~\ref{thm:spider_sufficient} when $n\geq 2k-1$.   

\begin{theorem}\label{thm:Wilsonian}
Let $Y$ be an $n$-vertex graph such that every $k$-vertex induced subgraph of $Y$ is connected. If $n\geq 2k-1$, then there exists a $k$-vertex induced subgraph $Y_0$ of $Y$ such that $\FS(\Star_{k},Y_0)$ is connected.
\end{theorem}

A \dfn{spider} is a connected graph that has a vertex $c$ such that deleting $c$ results in a disjoint union of paths. The vertex $c$ is called the \dfn{center} of the spider, and the paths that result from deleting $c$ are called the \dfn{legs} of the spider (the center is unique if there are at least three legs). The number of vertices in a leg is called its \dfn{length}. We write $\Spider(\lambda_1,\ldots,\lambda_k)$ for the spider with legs of lengths $\lambda_1,\ldots,\lambda_k$; note that the number of vertices in this graph is $\lambda_1+\cdots+\lambda_k+1$. The star graph $\Star_n$ is the spider with $n-1$ legs, all of which have length $1$. 

While Theorem~\ref{thm:spider_sufficient} provides a sufficient condition for the connectedness of $\FS(X,Y)$, the next theorem provides a necessary condition when $X$ is a spider. 

\begin{theorem}\label{thm:spider_necessary}
Let $\lambda_1\geq\cdots\geq\lambda_k$ be positive integers, and let $n=\lambda_1+\cdots+\lambda_k+1$. Let $Y$ be an $n$-vertex graph. If there exists a disconnected induced subgraph of $Y$ with $n-\lambda_1$ vertices, then $\FS(\Spider(\lambda_1,\ldots,\lambda_k),Y)$ is disconnected.  
\end{theorem}

The \dfn{dandelion} graph $\Dand_{k,n}$ is the spider graph with $k-1$ legs of length $1$ and $1$ leg of length $n-k$. In \cite{friends}, it was shown that the graph $\FS(\Dand_{2,n},Y)$ is connected if and only if $Y$ is the complete graph $K_n$ (note that $\Dand_{2,n}=\Path_n$) and that for $n\geq 5$, the graph $\FS(\Dand_{3,n},Y)$ is connected if and only if the minimum degree of $Y$ is at least $n-2$. The following corollary, which follows from Theorems~\ref{thm:spider_sufficient}, \ref{thm:Wilsonian}, and \ref{thm:spider_necessary}, substantially generalizes these results. 

\begin{corollary}\label{thm:dandelions}
If $n\geq 2k-1$, then $\FS(\Dand_{k,n},Y)$ is connected if and only if every induced subgraph of $Y$ with $k$ vertices is connected. 
\end{corollary}

Corollary~\ref{thm:dandelions} is noteworthy because it is rare to find families of graphs $X$ such that we can completely characterize all graphs $Y$ such that $\FS(X,Y)$ is connected. The next few theorems provide further characterizations of connectedness, though they put restrictions on both $X$ and $Y$. In what follows, we denote $\overline G$ to be the \dfn{complement} of $G$, which has the same vertex set as $G$ and satisfies $\{u,v\}\in E(\overline G)$ if and only if $\{u,v\}\not\in E(G)$. 

\begin{theorem}\label{thm:cycles_complement}
Let $\lambda_1\geq\cdots\geq\lambda_k$ be positive integers such that $\lambda_1+\cdots+\lambda_k+1=n\geq 4$. The friends-and-strangers graph $\FS(\Spider(\lambda_1,\ldots,\lambda_k),\overline{\Cycle_{n}})$ is connected if and only if $(\lambda_1,\ldots,\lambda_k)$ is not of the form $(\lambda_1,1,1)$ and is not in the following list:
    \[ (1,1,1,1),\quad (2,2,1),\quad  (2,2,2),\quad  (3,2,1),\quad  (3,3,1),\quad  (4,2,1),\quad  (5,2,1).\]
\end{theorem}

For $n\geq 4$, define the \dfn{fruit graph} $\Fruit_{n}$ to be the graph obtained from a cycle of size $n-1$ by adding an extra vertex and a single edge connecting that vertex to one of the vertices in the cycle. More precisely, $\Fruit_n$ has vertex set $[n]$ and edge set \[\{1,n-1\}\cup\{1,n\}\cup\{\{i,i+1\}:i\in[n-2]\}.\] The following result characterizes when the friends-and-strangers graph of a spider and the complement of a fruit graph is connected.

\begin{theorem}\label{thm:fruit_graphs}
Let $\lambda_1\geq\cdots\geq\lambda_k$ be positive integers such that $k\geq 3$ and $\lambda_1+\cdots+\lambda_k+1=n$. Then $\FS(\Spider(\lambda_1,\ldots,\lambda_k),\overline{\Fruit_n})$ is disconnected if and only if  $(\lambda_1,\ldots,\lambda_k)$ is of one of the following forms: \[(\lambda_1,1,1,1),\quad(\lambda_1,\lambda_2,1),\quad (2,2,2).\]
\end{theorem}

The next theorem guarantees the connectedness of $\FS(X,Y)$ whenever the minimum degree of $Y$ is large and $X$ is a connected graph that contains some small spider. 

\begin{theorem}\label{thm:spiders_min_degree}
Let $X$ and $Y$ be $n$-vertex graphs such that $Y$ has minimum degree at least $n-3$. The friends-and-strangers graph $\FS(X,Y)$ is connected if $X$ is connected and contains a (not necessarily induced) subgraph isomorphic to at least one of the following: 
\[\Star_7,\quad\Spider(2,1,1,1,1),\quad\Spider(2,2,1,1),\quad \Spider(3,3,2),\quad \Spider(4,2,2),\quad \Spider(4,3,1).\]
\end{theorem}

All of our results mentioned so far have been concerned with whether or not a friends-and-stranger graph is connected. In a different vein, it is also natural to study cycles in friends-and-strangers graphs; see \cite{Balitskiy, Felsner, naatz} for previous work on cycles in other flip graphs. Note that a cycle in a friends-and-strangers graph represents a nontrivial way that we can perform a sequence of friendly swaps that returns us to the arrangement of people with which we started.

A graph is called \dfn{even-degree} if each of its vertices has even degree. An \dfn{edge-subgraph} of a graph $G$ is a subgraph of $G$ that has the same vertex set as $G$. Given edge-subgraphs $H$ and $H'$ of $G$, let $H\triangle H'$ be the edge-subgraph of $G$ whose edge set is the symmetric difference of the edge sets of $H$ and $H'$. Note that if $H$ and $H'$ are even-degree, then so is $H\triangle H'$. The \dfn{cycle space} of $G$ is the set of all even-degree edge-subgraphs; it is a vector space over the $2$-element field $\mathbb F_2$ in which the addition operation is given by the symmetric difference $\triangle$. We can view a cycle in $G$ as an edge-subgraph in which all vertices not in the cycle are isolated vertices. It is well known that the cycle space of a (finite) graph is spanned by its cycles. Naatz \cite{naatz} proved that if $Y$ is any $n$-vertex graph, then the cycle space of $\FS(\Path_n,Y)$ is spanned by $4$-cycles and $6$-cycles (he stated this in the case when $Y$ is the incomparability graph of an $n$-element poset, but his methods apply more generally for any $Y$). 

In this article, we will study cycle spaces of friends-and-strangers graphs of the form $\FS(\Cycle_n,Y)$. The requisite analysis ends up being more complicated than what Naatz used to study $\FS(\Path_n,Y)$, but we will still obtain analogues of several of his results. 

A \dfn{dominating set} of a graph $G$ is a subset $D$ of the vertex set of $G$ such that every vertex in $G$ is either in $D$ or is adjacent to a vertex in $D$. The minimum size of a dominating set of $G$ is called the \dfn{domination number} of $G$. In the following theorem, we must impose the additional (fairly mild) condition that the domination number of $Y$ is at least $3$.

\begin{theorem}\label{thm:isometric}
Let $Y$ be an $n$-vertex graph with domination number at least $3$. The cycle space of $\FS(\Cycle_n,Y)$ is spanned by $4$-cycles and $6$-cycles. If $Y$ is triangle-free, then the cycle space of $\FS(\Cycle_n,Y)$ is spanned by $4$-cycles.
\end{theorem}

The proof of Theorem~\ref{thm:isometric} proceeds by first establishing a more general result (Theorem~\ref{thm:isocycles}) about when walks in $\FS(\Cycle_n,Y)$ can be obtained from one another via a sequence of \emph{Coxeter moves}, which are essentially the relations defining the affine symmetric group as a Coxeter group. 

As an example of Theorem~\ref{thm:isometric}, let $Y$ be the dandelion $\Dand_{3,8}$, which has domination number $3$ and is triangle-free. Figure~\ref{fig:cycles} shows one connected component of $\FS(\Cycle_8,\Dand_{3,8})$; upon inspection, we see that the cycle space of this connected component is generated by $4$-cycles, as predicted by Theorem~\ref{thm:isometric}. 

\begin{figure}[ht]
  \begin{center}\includegraphics[height=6cm]{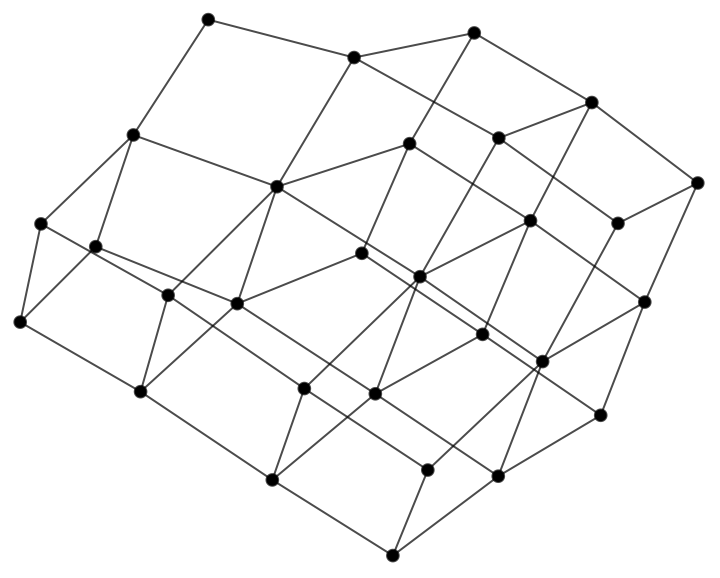}\end{center}
  \caption{One connected component of $\FS(\Cycle_8,\Dand_{3,8})$. }\label{fig:cycles}
\end{figure}

\begin{remark}
It would be interesting to understand the cycle spaces of other types of friends-and-strangers graphs besides those of the form $\FS(\Path_n,Y)$ or $\FS(\Cycle_n,Y)$. 
\end{remark}

The plan for the paper is as follows. In Section~\ref{sec:preliminaries}, we establish some notation and terminology, and we recall some previous results about friends-and-strangers graphs. 
In Section~\ref{Sec:General_Spiders}, we prove Theorems~\ref{thm:spider_sufficient},~\ref{thm:Wilsonian},~and \ref{thm:spider_necessary} and use them to deduce Corollary~\ref{thm:dandelions}. In Section~\ref{Sec:Complements_Cycles}, we establish Theorems~\ref{thm:cycles_complement},~\ref{thm:fruit_graphs}, and~\ref{thm:spiders_min_degree}. Section~\ref{Sec:Cycles} studies Coxeter moves and cycle spaces of friends-and-strangers graphs of the form $\FS(\Cycle_n,Y)$; it is in this section that we prove Theorem~\ref{thm:isometric}.

\section{Preliminaries}\label{sec:preliminaries}
The purpose of this section is to fix some terminology and notation and to discuss previous results that we will need in future sections.

\subsection{Transpositions and Walks}
For any set $S$ and any elements $s,s'\in S$, we write $(s\,\,s')$ for the bijection from $S$ to itself that swaps $s$ and $s'$ and fixes all elements of $S\setminus\{s,s'\}$. This is useful for notating friendly swaps in the friends-and-strangers graph $\FS(X,Y)$. Indeed, suppose $\sigma$ is a vertex in $\FS(X,Y)$ and $\{x,x'\}\in E(X)$. Let $y=\sigma(x)$ and $y'=\sigma(x')$. If $\{y,y'\}\in E(Y)$, then we can perform an $(X,Y)$-friendly swap across the edge $\{x,x'\}$ in order to obtain a new permutation $\sigma'=\sigma\circ (x\,\,x')=(y\,\,y')\circ\sigma$. 

A \dfn{walk} in a graph $G$ is a sequence of vertices such that any two consecutive vertices in the sequence are adjacent in $G$. Equivalently, we can think of a walk as a starting vertex together with a sequence of edges such that any two consecutive edges share a vertex. Each edge $e$ in the friends-and-strangers graph $\FS(X,Y)$ corresponds to an $(X,Y)$-friendly swap. It is convenient to label such an edge with the pair of people (i.e., vertices of $Y$) who performed the swap. More precisely, if $e=\{\sigma,\sigma'\}$, then we define the \dfn{edge label} $\psi(e)$ to be the unique pair $\{y,y'\}$ such that $\sigma'=(y\,\,y')\circ \sigma$. For ease of notation, we will often omit the set braces and write $yy'$ (or equivalently, $y'y$) for this edge label $\{y,y'\}$. Note that we can specify a walk in $\FS(X,Y)$ by saying its starting vertex together with it \dfn{label sequence}, which is just the sequence of edge labels of the edges used in the walk.

\subsection{Bipartite Graphs}
The following proposition from \cite{friends} tells us that the friends-and-strangers graph of two bipartite graphs is disconnected. 
\begin{proposition}[\cite{friends}]\label{prop:bipartite}
Suppose $X$ and $Y$ are bipartite graphs with $V(X)=V(Y)=[n]$. Let $\{A_X,B_X\}$ be a bipartition of $X$, and let $\{A_Y,B_Y\}$ be a bipartition of $Y$. For $\sigma\in V(\FS(X,Y))$, define $p(\sigma)\in\{0,1\}$ by \[p(\sigma)=|\sigma(A_X)\cap A_Y|+\frac{\sgn(\sigma)+1}{2}\pmod{2},\] where $\sgn(\sigma)$ is the sign of $\sigma$, viewed as a permutation in the symmetric group $S_n$. If $\tau,\tau'\in V(\FS(X,Y))$ are such that $p(\tau)\neq p(\tau')$, then $\tau$ and $\tau'$ are in different connected components of $\FS(X,Y)$.  
\end{proposition}

\subsection{Wilson's Theorem}\label{subsec:Wilson}
Given a graph $G$ and a vertex $v\in V(G)$, we write $G\setminus v$ for the graph obtained from $G$ by deleting $v$ (and all edges incident to $v$). We say $v$ is a \dfn{cut vertex} of $G$ if $G$ is connected and $G\setminus v$ is disconnected. We say $G$ is \dfn{biconnected} if it is connected and has no cut vertices. Wilson noted that if $n\geq 3$ and $Y$ is an $n$-vertex graph that is not biconnected, then $\FS(\Star_n,Y)$ is disconnected. Thus, he focused his attention on friends-and-strangers graphs of stars and biconnected graphs. The statement of his main theorem involves the exceptional graph 
\[\theta_0=
\begin{array}{l}
\begin{tikzpicture}[line width=1pt, x=0.15cm,y=0.15cm]
\tikzstyle{dot}=[circle,thin,draw,fill=black,inner sep=1.5pt]
\pgfdeclarelayer{nodelayer}
\pgfdeclarelayer{edgelayer}
\pgfsetlayers{edgelayer,nodelayer}
	\begin{pgfonlayer}{nodelayer}
		\node [style=dot] (0) at (0, 0) {};
		\node [style=dot] (1) at (6, 0) {};
		\node [style=dot] (2) at (-3, 5.2) {};
		\node [style=dot] (3) at (0, 10.4) {};
		\node [style=dot] (4) at (6, 10.4) {};
		\node [style=dot] (5) at (9, 5.2) {};
		\node [style=dot] (6) at (3, 5.2) {};
		\node [style=dot] (7) at (-3, 5.2) {};
	\end{pgfonlayer}
	\begin{pgfonlayer}{edgelayer}
		\draw (7) to (6);
		\draw (0) to (7);
		\draw (7) to (3);
		\draw (3) to (4);
		\draw (4) to (5);
		\draw (5) to (6);
		\draw (0) to (1);
		\draw (1) to (5);
	\end{pgfonlayer}
\end{tikzpicture}
\end{array}.
\]

\begin{theorem}[\cite{wilson}]\label{thm:wilson}
Let $Y$ be a biconnected graph on $n\geq3$ vertices that is not isomorphic to $\theta_0$ or $\Cycle_n$. If $Y$ is not bipartite, then $\FS(\Star_n, Y )$ is connected. If $Y$ is bipartite, then $\FS(\Star_n, Y )$ has
exactly $2$ connected components, each with $n!/2$ vertices. The graph $\FS(\Star_7, \theta_0 )$ has exactly $6$ connected
components.
\end{theorem}

\begin{remark}\label{rem:Star_Cycle}
It is straightforward to describe the connected components of $\FS(\Star_n,\Cycle_n)$. For each vertex $\sigma$ of $\FS(\Star_n,\Cycle_n)$, we can read off the leaves of $\Star_n$ in the clockwise order that their images under $\sigma$ appear around $\Cycle_n$. This defines a cyclic ordering of the set of leaves of $\Star_n$. It is straightforward to see that two permutations are in the same connected component of $\FS(\Star_n,\Cycle_n)$ if and only if they induce the same cyclic ordering of the leaves of $\Star_n$. 
\end{remark}

\section{General Theorems on Spiders}\label{Sec:General_Spiders}

Our goal in this section is to prove Theorems~\ref{thm:spider_sufficient}, \ref{thm:Wilsonian}, and \ref{thm:spider_necessary}, which together imply Corollary~\ref{thm:dandelions}. We begin with Theorem~\ref{thm:Wilsonian} since its proof does not require any further preliminary results. 

\begin{proof}[Proof of Theorem~\ref{thm:Wilsonian}]
Suppose $n\geq 2k-1$, and let $Y$ be an $n$-vertex graph such that every $k$-vertex induced subgraph of $Y$ is connected. We will use induction on $\ell$ to show that for each $\ell\in\{3,\ldots,k\}$, there exists a biconnected $\ell$-vertex induced subgraph $Y^{(\ell)}$ of $Y$ that contains a triangle. The proof will then follow by setting $Y_0=Y^{(k)}$. Indeed, if $\ell=3$, then $\FS(\Star_3,Y_0)=\FS(\Star_3,K_3)$ is connected, and if $\ell\geq 4$, then Theorem~\ref{thm:wilson} guarantees that $\FS(\Star_k,Y_0)$ is connected (since $Y^{(k)}$ is biconnected, not bipartite, not a cycle, and not $\theta_0$). 

Because every $k$-vertex induced subgraph of $Y$ is connected, we see that every vertex in $Y$ has degree at least $n-k+1$. The number of edges in $Y$ is half the sum of the degrees of the vertices in $Y$, which is at least $n(n-k+1)/2>n^2/4$. Therefore, it follows from Tur\'an's theorem that $Y$ contains a triangle $Y^{(3)}$. This completes the base case of our induction. 

Now suppose $\ell\in\{4,\ldots,k\}$, and assume inductively that there exists a biconnected $(\ell-1)$-vertex induced subgraph $Y^{(\ell-1)}$ of $Y$ that contains a triangle. We claim that the number of edges between vertices in $V(Y^{(\ell-1)})$ and vertices in $V(Y)\setminus V(Y^{(\ell-1)})$ is at least $n-\ell+2$. Since $|V(Y)\setminus V(Y^{(\ell-1)})|=n-\ell+1$, it will then follow from the pigeonhole principle that there exists a vertex $v\in V(Y)\setminus V(Y^{(\ell-1)})$ that has at least two neighbors in $Y^{(\ell-1)}$. We can then simply set $Y^{(\ell)}$ to be the induced subgraph of $Y$ whose vertex set is $V(Y^{(\ell-1)})\cup\{v\}$; it is easy to see that this induced subgraph is biconnected and contains a triangle (the same triangle as in $Y^{(\ell-1)}$). 

To prove that claim, we use the fact, which we mentioned before, that each vertex in $Y$ has degree at least $n-k+1$. Since $Y^{(\ell-1)}$ has $\ell-1$ vertices, this implies that each vertex in $Y^{(\ell-1)}$ has at least $n-k+1-(\ell-2)=n-k-\ell+3$ neighbors in $V(Y)\setminus V(Y^{(\ell-1)})$. it follows that there are at least $(\ell-1)(n-k-\ell+3)$ edges between vertices in $V(Y^{(\ell-1)})$ and vertices in $V(Y)\setminus V(Y^{(\ell-1)})$. Thus, to prove the claim, we must show that $(\ell-1)(n-k-\ell+3)\geq n-\ell+2$. This inequality rearranges to $p(\ell)\geq 0$, where $p(x)=-x^2+(n-k+5)x-(2n-k+5)$. We will show that $p(3)\geq 0$ and $p(k)\geq 0$; since $p(x)$ is a quadratic polynomial and $3\leq \ell\leq k$, this will imply that $p(\ell)\geq 0$. First, we have $p(3)=n-2k+1\geq 0$ by hypothesis. On the other hand, $p(k)=q(k)$, where $q(x)$ is the quadratic polynomial $-2x^2+(n+6)x-(2n+5)$. To see that $q(k)\geq 0$, we use the fact that $3\leq k\leq (n+1)/2$ and compute that $q(3)=n-5\geq 0$ and $q((n+1)/2)=(n-5)/2\geq 0$.
\end{proof}

We now move on to proving Theorems~\ref{thm:spider_sufficient}, and \ref{thm:spider_necessary}, for which we need the following lemmas. 

\begin{lemma}\label{thm:spider_sufficient_helper1}
Fix an integer $n \geq 3$, and let $Y$ be a biconnected graph with $n$ vertices. Let $\sigma$ be a vertex of $\FS(\Star_n,Y)$, and fix $x\in V(\Star_n)$ and $y\in V(Y)$. There exists a vertex $\sigma'$ in the same connected component of $\FS(\Star_n,Y)$ as $\sigma$ such that $\sigma'(x) = y$.

\end{lemma}
\begin{proof}
If $Y$ is a cycle, then the desired result follows easily from Remark~\ref{rem:Star_Cycle}. If $n=7$ and $Y$ is the exceptional graph $\theta_0$ shown in Section~\ref{subsec:Wilson}, then we can check by hand that the desired result still holds. Now suppose $Y$ is not a cycle or $\theta_0$. If $Y$ is not bipartite, then Theorem~\ref{thm:wilson} tells us that $\FS(X,Y)$ is connected, so the desired result is obvious. 

Now suppose $Y$ is bipartite. Since $Y$ is biconnected, we must have $n\geq 4$. We may assume $V(\Star_n)=V(Y)=[n]$; let $\{A_Y,B_Y\}$ be a bipartition of $Y$. We assume that $n$ is the center of $\Star_n$ so that $\{[n-1],\{n\}\}$ is a bipartiton of $\Star_n$. For each $\mu\in V(\FS(\Star_n,Y))$, let $p(\mu)=|\mu([n-1])\cap A_Y|+\dfrac{\sgn(\mu)+1}{2}\pmod{2}$, as in Proposition~\ref{prop:bipartite}. Choose $u,v\in [n-1]\setminus\{x\}$. Let $\tau$ be a vertex in $\FS(X,Y)$ such that $\tau(x)=y$, and let $\tau'=\tau\circ(u\,\,v)$. Then $\tau'(x)=y$. We have $p(\tau)\neq p(\tau')$, so it follows from Proposition~\ref{prop:bipartite} that $\tau$ and $\tau'$ are in different connected components of $\FS(\Star_n,Y)$. Theorem~\ref{thm:wilson} tells us that $\FS(\Star_n,Y)$ has exactly $2$ connected components, so we can take $\sigma'$ to be whichever of the vertices $\tau$ or $\tau'$ is in the same connected component as $\sigma$. 
\end{proof}

\begin{lemma}\label{thm:spider_sufficient_helper2}
Let $X$ and $Y$ be connected $n$-vertex graphs such that $X$ has maximum degree $k\geq 2$. Suppose every induced subgraph of $Y$ with $k$ vertices is connected. Let $\sigma$ be a vertex of $\FS(X,Y)$, and fix $x\in V(X)$ and $y\in V(Y)$. There exists a vertex $\sigma'$ in the same connected component of $\FS(X,Y)$ as $\sigma$ such that $\sigma'(x) = y$. 
\end{lemma}
\begin{proof}

If $T$ is a spanning tree of $X$, then $\FS(T,Y)$ is a subgraph of $\FS(X,Y)$. Therefore, it suffices to prove the lemma when $X$ is a tree; we assume that this is the case in what follows. Note that the result is trivial if $\sigma(x)=y$. Since $Y$ is connected, it suffices to prove the result when $\{\sigma(x),y\}$ is an edge in $Y$; we assume that this is the case in what follows. We proceed by induction on $n$. 

Let $c$ be a vertex of $X$ of degree $k$, and let $N[c]$ denote the closed neighborhood of $c$ (i.e., the set of vertices that are adjacent to or equal to $c$). If $X$ is a star, then $n=k+1$, so the hypothesis that every $k$-vertex induced subgraph of $Y$ is connected is equivalent to the fact that $Y$ is biconnected. Thus, the desired result follows immediately from Lemma~\ref{thm:spider_sufficient_helper1}. Note that this handles the base case $n=3$. In what follows, we may assume $X$ is not a star. We consider three cases. 

\medskip
\noindent {\bf Case 1.} Suppose $X$ has a leaf $u$ such that $u\not\in N[c]$, $u\neq x$, and $u\neq\sigma^{-1}(y)$. Let $X'=X\setminus u$ and $Y'=Y\setminus \sigma(u)$. Let $\sigma\vert_{X'}$ be the restriction of $\sigma$ to $V(X')$. Since $u$ is a leaf of $X$ and $N[c]\subseteq V(X')$, the graph $X'$ is a tree with $n-1$ vertices and maximum degree $k$. Furthermore, every $k$-vertex induced subgraph of $Y'$ is connected.  Therefore, it follows by induction that there exists a sequence of $(X',Y')$-friendly swaps that transforms $\sigma\vert_{X'}$ into a vertex $\mu$ of $\FS(X',Y')$ such that $\mu(x)=y$. Let $\sigma'$ be the unique vertex in $\FS(X,Y)$ such that $\sigma'(u)=\sigma(u)$ and $\sigma'(v)=\mu(v)$ for all $v\in V(X')$. Then $\sigma'(x)=y$. The sequence of $(X',Y')$-friendly swaps transforming $\sigma\vert_{X'}$ into $\mu$ can be interpreted as a sequence of $(X,Y)$-friendly swaps that transforms $\sigma$ into $\sigma'$. Thus, $\sigma'$ is in the same connected component of $\FS(X,Y)$ as $\sigma$. 

\medskip
\noindent {\bf Case 2.} Suppose that $x$ is a leaf of $X$ and $x\not\in N[c]$. Let $x'$ be the neighbor of $x$ in $X$. Let $X'=X\setminus x$ and $Y'=Y\setminus \sigma(x)$. Let $\sigma\vert_{X'}$ be the restriction of $\sigma$ to $V(X')$. The graph $X'$ is a tree with $n-1$ vertices and maximum degree $k$, and every $k$-vertex induced subgraph of $Y'$ is connected.  It follows by induction (and the assumption $\sigma(x)\neq y$) that there exists a sequence of $(X',Y')$-friendly swaps that transforms $\sigma\vert_{X'}$ into a vertex $\mu$ of $\FS(X',Y')$ such that $\mu(x')=y$. Let $\tau$ be the unique vertex in $\FS(X,Y)$ such that $\tau(x)=\sigma(x)$ and $\tau(v)=\mu(v)$ for all $v\in V(X')$. In particular, $\tau(x')=y$. The sequence of $(X',Y')$-friendly swaps transforming $\sigma\vert_{X'}$ into $\mu$ can be interpreted as a sequence of $(X,Y)$-friendly swaps that transforms $\sigma$ into $\tau$. Thus, $\tau$ is in the same connected component of $\FS(X,Y)$ as $\sigma$. Now let $\sigma'=\tau\circ(x\,\,x')$. Since $\{x,x'\}\in E(X)$ and $\{\tau(x),\tau(x')\}=\{\sigma(x),y\}\in E(Y)$, the vertices $\tau$ and $\sigma'$ are adjacent in $\FS(X,Y)$. Then $\sigma'$ is in the same connected component as $\sigma$ and satisfies $\sigma'(x)=y$. 

\medskip
\noindent {\bf Case 3.} Suppose that there is a leaf $u$ of $X$ such that $u\not\in N[c]$ and $\sigma(u)=y$. We have assumed that $\sigma(x)\neq y$, so $u\neq x$. Let $y'$ be a neighbor of $y$ in $Y$. Since $\{\sigma(u),y'\}$ is an edge in $Y$, we can repeat the argument in Case 2 with $x$ replaced by $u$ and $y$ replaced by $y'$; this allows us to deduce that there is a vertex $\widetilde\sigma$ of $\FS(X,Y)$ in the same connected component as $\sigma$ such that $\widetilde\sigma(u)=y'$. But now observe that $u$ is a leaf of $X$ such that $u\not\in N[c]$, $u\neq x$, and $u\neq\sigma^{-1}(y)$; this means that we can repeat the argument from Case 1 with $\sigma$ replaced by $\widetilde\sigma$ to see that there is a vertex $\sigma'$ in the same connected component of $\FS(X,Y)$ as $\widetilde\sigma$ such that $\sigma'(x)=y$. But then $\sigma'$ is also in the same connected component of $\FS(X,Y)$ as $\sigma$.
\end{proof}

We can now complete the proofs of Theorems~\ref{thm:spider_sufficient} and~\ref{thm:spider_necessary}. 

\begin{proof}[Proof of Theorem~\ref{thm:spider_sufficient}]
Let $n,k,X,Y,Y_0$ be as in the statement of the theorem; we want to prove that $\FS(X,Y)$ is connected. If $T$ is a spanning tree of $X$, then $\FS(T,Y)$ is a subgraph of $\FS(X,Y)$. Therefore, we may assume in what follows that $X$ is a tree. We may also assume that $n\geq 4$ since the case when $n=3$ can be checked by hand.

First, suppose $n=k+1$. In this case, $X$ is isomorphic to $\Star_n$. The hypothesis that every $k$-vertex induced subgraph of $Y$ is connected is equivalent to the fact that $Y$ is biconnected. Since $\FS(\Star_k,Y_0)$ is connected, it follows from the discussion in Section~\ref{subsec:Wilson} that the graph $Y$ cannot be a cycle or the exception graph $\theta_0$; indeed, if it were, then $Y_0$ would either not be biconnected or would be a cycle. By Theorem~\ref{thm:wilson} (or Proposition~\ref{prop:bipartite}), the hypothesis that $\FS(\Star_k,Y_0)$ is connected guarantees that $Y_0$ is not bipartite, which implies that $Y$ is not bipartite. Hence, it follows from Theorem~\ref{thm:wilson} that $\FS(\Star_n,Y)$ is connected, as desired. 

We may now assume $n\geq k+2$ and proceed by induction on $n$. Let $c$ be a vertex of $X$ of degree $k$, and let $N[c]$ be the closed neighborhood of $c$. There must be a leaf $x$ of $X$ such that $x\not\in N[c]$. Fix some vertex $y\in V(Y)\setminus\{Y_0\}$. Fix a vertex $\tau$ of $\FS(X,Y)$ such that $\tau(x)=y$. Choose some vertex $\sigma$ of $\FS(X,Y)$. We will show that $\sigma$ is in the same connected component as $\tau$; as $\sigma$ was arbitrary, this will prove that $\FS(X,Y)$ is connected. 

According to Lemma~\ref{thm:spider_sufficient_helper2}, there is a vertex $\sigma'$ in the same connected component of $\FS(X,Y)$ as $\sigma$ such that $\sigma'(x)=y$. Let $X'=X\setminus x$ and $Y'=Y\setminus y$. Let $\sigma'\vert_{X'}$ and $\tau\vert_{X'}$ be the restrictions of $\sigma'$ and $\tau$, respectively, to $V(X')$. Note that $X'$ is a tree with $n-1$ vertices and maximum degree $k$. Furthermore, every $k$-vertex induced subgraph of $Y'$ is connected, and $Y'$ contains the induced subgraph $Y_0$ such that $\FS(\Star_k,Y)$ is connected. Thus, we can use induction to see that $\FS(X',Y')$ is connected. This means that there is a sequence of $(X',Y')$-friendly swaps that transforms $\sigma'\vert_{X'}$ into $\tau\vert_{X'}$. This sequence of $(X',Y')$-friendly swaps can be interpreted as a sequence of $(X,Y)$-friendly swaps that transforms $\sigma'$ into $\tau$. This shows that $\sigma'$ is in the same connected component of $\FS(X,Y)$ as $\tau$, so $\sigma$ is also in the same connected component as $\tau$. 
\end{proof}

\begin{proof}[Proof of Theorem~\ref{thm:spider_necessary}]
Let $\lambda_1\geq\cdots\geq\lambda_k$ be as in the statement of the theorem, and let $X=\Spider(\lambda_1,\ldots,\lambda_k)$. Let $z$ be the leaf of $X$ on the leg of length $\lambda_1$. Let us define a partial order $\preceq$ on $V(X)$ by declaring that $x\preceq x'$ if the unique path from $z$ to $x'$ contains $x$. 

By hypothesis, there is a disconnected induced subgraph $Y_0$ of $Y$ with $n-\lambda_1$ vertices. There exist nonempty subsets $A,B\subseteq V(Y_0)$ such that $A\cup B= V(Y_0)$ and $A\cap B=\emptyset$ and such that no vertex in $A$ is adjacent in $Y$ to any vertex in $B$. Let us say a vertex $\sigma\in\FS(X,Y)$ is \dfn{special} if there exists $a_0\in A$ such that $\sigma^{-1}(a_0)\preceq\sigma^{-1}(b)$ for all $b\in B$. The graph $\FS(X,Y)$ has at least one special vertex and at least one non-special vertex. Therefore, in order to prove that $\FS(X,Y)$ is disconnected, it suffices to show that every connected component of $\FS(X,Y)$ consists entirely of special vertices or entirely of non-special vertices. 

Suppose instead that there is a connected component of $\FS(X,Y)$ that contains both special and non-special vertices. Then this connected component must have a special vertex $\sigma$ that is adjacent to a non-special vertex $\tau$. Let $\{u,v\}\in E(Y)$ be the edge label of the edge $\{\sigma,\tau\}\in E(\FS(X,Y))$; that is, $\tau=\sigma\circ(u\,\,v)$. Then $\{\sigma^{-1}(u),\sigma^{-1}(v)\}\in E(X)$. Because $\sigma$ is special, there exists $a_0\in A$ such that $\sigma^{-1}(a_0)\preceq\sigma^{-1}(b)$ for all $b\in B$. Using the definition of the partial order $\preceq$ and the fact that $X$ is a spider, it is now straightforward to see that $a_0\in\{u,v\}$; without loss of generality, we may assume $a_0=u$. Note that $\tau^{-1}(v)=\sigma^{-1}(a_0)\preceq\sigma^{-1}(b)=\tau^{-1}(b)$ for all $b\in B$; if $v$ were in $A$, then this would contradict the fact that $\tau$ is not special. This shows that $v\not\in A$, and we also know that $v\not\in B$ since $\{u,v\}\in E(Y)$ and no vertex in $A$ is adjacent to any vertex in $B$. Thus, $v\in V(Y)\setminus(A\cup B)$. 

Let $L$ be the set of vertices in $X$ that are in the leg of length $\lambda_1$ containing $z$. Thus, $|L|=\lambda_1$. Because $\sigma$ is special and $\tau$ is not, there exists $b_0\in B$ such that $\sigma^{-1}(a_0)\preceq \sigma^{-1}(b_0)$ while $\tau^{-1}(a_0)\not\preceq\tau^{-1}(b_0)$. This forces $\sigma^{-1}(a_0)$ to be the center of the spider $X$, and it also forces $\sigma^{-1}(v)$ to be in $V(X)\setminus L$. Since $\sigma^{-1}(a_0)\preceq\sigma^{-1}(b)$ for all $b\in B$, we must have $\tau^{-1}(B)=\sigma^{-1}(B)\subseteq V(X)\setminus L$. If there were a vertex $a_1\in\tau^{-1}(A)\cap L$, then we would have $\tau^{-1}(a_1)\preceq\tau^{-1}(b)$ for all $b\in B$, contradicting the fact that $\tau$ is not special. Hence, $\tau^{-1}(A)\subseteq V(X)\setminus L$. We also know that $\tau^{-1}(v)=\sigma^{-1}(a_0)$ is the center of $X$, so $\tau^{-1}(v)\in V(X)\setminus L$. Thus, $\tau^{-1}(A)\cup\tau^{-1}(B)\cup\{\tau^{-1}(v)\}$ is a set of size $n-\lambda_1+1$ that is contained in $V(X)\setminus L$. This is a contradiction because $|V(X)\setminus L|=n-\lambda_1$.   
\end{proof}

\section{Spiders, Cycles, and Fruits}\label{Sec:Complements_Cycles}

The goal of this section is to prove Theorems~\ref{thm:cycles_complement},~\ref{thm:fruit_graphs}, and~\ref{thm:spiders_min_degree}. The main tool for proving these theorems is the following lemma, which will allow us to build connected friends-and-strangers graphs from smaller ones. A family $\mathcal Y$ of (isomorphism classes of) graphs is called \dfn{hereditary} if it is closed under taking induced subgraphs (i.e., every induced subgraph of a graph in $\mathcal Y$ is also in $\mathcal Y$). 

\begin{lemma}\label{lem:hereditary}
Let $\mathcal Y$ be a hereditary family of connected graphs. Let $X$ be a graph with $n$ vertices such that $\FS(X,Y)$ is connected for every $n$-vertex graph $Y\in\mathcal Y$. Let $x\in V(X)$, and let $X'$ be the graph obtained from $X$ by adding a new vertex $x'$ together with the edge $\{x,x'\}$. Then $\FS(X',Y')$ is connected for every $(n+1)$-vertex graph $Y'\in \mathcal Y$. 
\end{lemma}

\begin{proof}
We will prove that if $\sigma,\sigma'\colon V(X')\to V(Y')$ are two vertices of $\FS(X',Y')$ such that $\sigma(x')$ and $\sigma'(x')$ are adjacent in $Y'$, then $\sigma$ and $\sigma'$ are in the same connected component of $\FS(X',Y')$. Since $Y'$ is connected, this will imply that any two vertices in $\FS(X',Y')$ is connected.

Let $y=\sigma(x')$ and $y'=\sigma'(x')$, and suppose $\{y,y'\}\in E(Y')$. Let $F_y$ be the subgraph of $\FS(X',Y')$ induced by the set of vertices $\tau$ satisfying $\tau(x')=y$. Similarly, define $F_{y'}$ to be the subgraph of $\FS(X',Y')$ induced by the set of vertices $\tau'$ satisfying $\tau'(x') = y'$. Let $Y_y$ and $Y_{y'}$ be the induced subgraphs of $Y'$ on the vertex sets $V(Y')\setminus\{y\}$ and $V(Y')\setminus\{y'\}$, respectively. It is clear that $F_y$ and $F_{y'}$ are isomorphic to $\FS(X,Y_y)$ and $\FS(X,Y_{y'})$, respectively. Since $\mathcal Y$ is hereditary, we know that $Y_y$ and $Y_{y'}$ are both in $\mathcal Y$. Therefore, our hypothesis on $X$ guarantees that both $F_y$ and $F_{y'}$ are connected graphs.

Let $\mu$ be any vertex in $\FS(X',Y')$ satisfying $\mu(x)=y'$ and $\mu(x') = y$, and let $\mu'$ be the vertex $\mu \circ (x \; x')$. Note that $\mu'(x') = y'$ and that $\mu$ and $\mu'$ are adjacent in $\FS(X,Y)$. We have $\mu \in F_y$ and $\mu' \in F_{y'}$, so there are paths in $\FS(X',Y')$ from $\sigma$ to $\mu$ and from $\sigma'$ to $\mu'$. Therefore, $\sigma$ and $\sigma'$ are in the same connected component of $\FS(X',Y')$, as desired.
\end{proof}

\begin{corollary}\label{lemma2.2} 
Let $X$ be a graph with $n\geq 5$ vertices such that $\FS(X,\overline{\Cycle_n})$ is connected. Let $x\in V(X)$, and let $X'$ be the graph obtained from $X$ by adding a new vertex $x'$ together with the edge $\{x,x'\}$. Then $\FS(X',\overline{\Cycle_{n+1}})$ is connected. 
\end{corollary}

\begin{proof}
Let $\mathcal Y$ be the smallest hereditary family of graphs that contains $\overline{\Cycle_N}$ for all $N\geq 5$. In other words, $\mathcal Y$ is the collection of graphs that can be realized as induced subgraphs of complements of cycles. It is straightforward to verify that a graph is in $\mathcal Y$ if and only if its complement is a cycle or a disjoint union of paths. It follows that if $Y$ is an $n$-vertex graph in $\mathcal Y$, then $\overline{\Cycle_n}$ is a subgraph of $Y$. This shows that $\FS(X,Y)$ is connected for every $n$-vertex graph $Y$ in $\mathcal Y$, so it follows from Lemma~\ref{lem:hereditary} that $\FS(X',Y')$ is connected for every $(n+1)$-vertex graph $Y'$ in $\mathcal Y$. In particular, $\FS(X',\overline{\Cycle_{n+1}})$ is connected. 
\end{proof}

\begin{corollary}\label{lemma2.2fruit} 
Let $X$ be a graph with $n\geq 5$ vertices such that $\FS(X,\overline{\Cycle_n})$ and $\FS(X,\overline{\Fruit_n})$ are both connected. Let $x\in V(X)$, and let $X'$ be the graph obtained from $X$ by adding a new vertex $x'$ together with the edge $\{x,x'\}$. Then $\FS(X',\overline{\Fruit_{n+1}})$ is connected. 
\end{corollary}

\begin{proof}
Let $\mathcal Y$ be the smallest hereditary family of graphs that contains $\overline{\Fruit_N}$ for every $N\geq 5$. In other words, $\mathcal Y$ is the collection of graphs that can be realized as induced subgraphs of complements of fruit graphs. It is straightforward to verify that every $n$-vertex graph in $\mathcal Y$ contains either $\overline{\Cycle_n}$ or $\overline{\Fruit_n}$ as a subgraph. This shows that $\FS(X,Y)$ is connected for every $n$-vertex graph $Y$ in $\mathcal Y$, so it follows from Lemma~\ref{lem:hereditary} that $\FS(X',Y')$ is connected for every $(n+1)$-vertex graph $Y'$ in $\mathcal Y$. In particular, $\FS(X',\overline{\Fruit_{n+1}})$ is connected. 
\end{proof}

\begin{corollary}\label{cor:min_degre}
Let $X$ be a graph with $n\geq 4$ vertices such that $\FS(X,Y)$ is connected for all $n$-vertex graphs $Y$ with minimum degree at least $n-3$. Let $x\in V(X)$, and let $X'$ be the graph obtained from $X$ by adding a new vertex $x'$ together with the edge $\{x,x'\}$. If $Y$ is a graph with $n+1$ vertices and minimum degree at least $n-2$, then $\FS(X',Y')$ is connected. 
\end{corollary}

\begin{proof}
This follows immediately from Lemma~\ref{lem:hereditary} by setting $\mathcal Y$ to be the collection of graphs whose complements have maximum degree at most $2$. 
\end{proof}

We can now proceed to the proofs of the main theorems of this section. 

\begin{proof}[Proof of Theorem~\ref{thm:cycles_complement}]
Choose $\lambda_1\geq\cdots\geq\lambda_k\geq 1$ with $k\geq 3$ and $\lambda_1\cdots+\lambda_k+1=n$. We can easily check using a computer that $\FS(\Spider(\lambda_1,\ldots,\lambda_k),\overline{\Cycle_n})$ is disconnected if $(\lambda_1,\ldots,\lambda_k)$ is one of the seven partitions listed in the statement of the theorem. Now suppose $(\lambda_1,\ldots,\lambda_k)=(\lambda_1,1,1)$. We know by \cite[Theorem~6.5]{friends} that $\FS(\Spider(\lambda_1,1,1),Y)$ is connected if and only if the minimum degree of $Y$ is at least $n-2$; since the minimum degree of $\overline{\Cycle_n}$ is $n-3$, $\FS(\Spider(\lambda_1,1,1),\overline{\Cycle_n})$ is disconnected. 

Now assume $(\lambda_1,\ldots,\lambda_k)$ is not of the form $(\lambda_1,1,1)$ and is not one of the seven partitions listed in the statement of Theorem~\ref{thm:cycles_complement}. It is straightforward to check that these assumptions guarantee that $\Spider(\lambda_1,\ldots,\lambda_k)$ contains an induced subgraph $\Spider(\rho_1,\ldots,\rho_r)$, where $(\rho_1,\ldots,\rho_r)$ is one of the following partitions: 
\[(1,1,1,1,1),\quad (2,1,1,1),\quad (6,2,1),\quad (4,3,1),\quad (3,2,2).\] Let $n'=\rho_1+\cdots+\rho_r+1$. We can verify by computer that $\FS(\Spider(\rho_1,\ldots,\rho_r),\overline{\Cycle_{n'}})$ is connected. The graph $\Spider(\lambda_1,\ldots,\lambda_k)$ can be obtained from $\Spider(\rho_1,\ldots,\rho_r)$ by a sequence of operations, where each operation adds a new vertex to the graph and adds a new edge that has the new vertex as one of its endpoints. Therefore, it follows from Corollary~\ref{lemma2.2} that $\FS(\Spider(\lambda_1,\ldots,\lambda_k),\overline{\Cycle_n})$ is connected. 
\end{proof}

\begin{proof}[Proof of Theorem~\ref{thm:fruit_graphs}]
We can use a computer to check that \[\FS(\Spider(1,1,1,1,1),\overline{\Fruit_6}),
\quad \FS(\Spider(2,2,1,1),\overline{\Fruit_7}),\quad\text{and}\quad\FS(\Spider(2,2,3),\overline{\Fruit_8})\] are connected. It now follows easily from Corollary~\ref{lemma2.2fruit} that the graph $\FS(\Spider(\lambda_1,\ldots,\lambda_k),\overline{\Fruit_n})$ is connected whenever $k\geq 5$ or $k=4$ and $\lambda_2\geq 2$ or $k=3$ and $\lambda_1\geq3$ and $\lambda_3\geq 2$. 

With a computer, we can also verify that $\FS(\Spider(2,2,2),\overline{\Fruit_7})$ has 12 connected components. Thus, to prove the reverse direction, we must show that $\FS(\Spider(\lambda_1,\ldots,\lambda_k),\overline{\Fruit_n})$ is disconnected if $k=3$ and $\lambda_3 = 1$ or if $k=4$ and $\lambda_2=1$.  Since $V(\overline{\Fruit_n})=V(\Fruit_n)$, we can consider the identity bijection $\text{id}\colon V(\Fruit_n)\to V(\overline{\Fruit_n})$; it is straightforward to see that this bijection is an isolated vertex in $\FS(\Fruit_n,\overline{\Fruit_n})$. Hence, $\FS(\Fruit_n,\overline{\Fruit_n})$ is disconnected. It is also straightforward to see that $\Spider(\lambda_1,\lambda_2,1)$ is a subgraph of $\Fruit_{n}$, so this resolves the case where $k = 3$ and $\lambda_3 = 1$. Now suppose $k = 4$ and $\lambda_2 = 1$. Then $\Spider(\lambda_1,1,1,1)$ is the dandelion $\Dand_{4,n}$, and $n-\lambda_1=4$. Since $\Fruit_n$ clearly contains a disconnected induced subgraph with $4$ vertices, it follows from Theorem~\ref{thm:spider_necessary} that $\FS(\Spider(\lambda_1,1,1,1),\Fruit_n)$ is disconnected.
\end{proof}

\begin{proof}[Proof of Theorem~\ref{thm:spiders_min_degree}]
Let $G$ be one of the six graphs listed in the statement of Theorem~\ref{thm:spiders_min_degree}, and let $X$ be a connected graph that contains $G$ as a subgraph. Let $n$ and $n'$ be the number of vertices of $X$ and $G$, respectively. We can check by computer that $\FS(G,H)$ is connected for every graph $H$ with $n'$ vertices and minimum degree at least $n'-3$. Because $X$ is connected, it is possible to obtain an $n$-vertex subgraph $\widetilde X$ of $X$ by sequence of operations, where each operation adds a new vertex to the graph and adds a new edge that has the new vertex as one of its endpoints. It follows from Corollary~\ref{cor:min_degre} that if $Y$ is an $n$-vertex graph with minimum degree at least $n-3$, then $\FS(\widetilde X,Y)$ is connected. Since $\widetilde X$ is a subgraph of $X$ with the same vertex set as $X$, the friends-and-strangers graph $\FS(X,Y)$ must also be connected for every such $Y$. 
\end{proof}

\section{Coxeter Moves and the Cycle Space}\label{Sec:Cycles}

Recall from Section~\ref{SecIntro} that the \dfn{cycle space} of a graph $G$ is the vector space over $\mathbb F_2$ whose elements are the even-degree edge-subgraphs of $G$ and whose addition is the symmetric difference operation $\triangle$. This vector space is spanned by the cycles of $G$. The goal of this section is to prove Theorem~\ref{thm:isometric}, which provides smaller generating sets for the cycle spaces of graphs of the form $\FS(\Cycle_n,Y)$ when $Y$ has domination number at least $3$. 

We provide a framework to study paths in $\FS(\Cycle_n, Y)$ that is inspired by Coxeter moves in the affine symmetric group. Under this framework, we also prove some results that apply when $Y$ has domination number $2$. 

Recall the definitions of \emph{walks}, \emph{edge labels}, and \emph{label sequences} from Section~\ref{sec:preliminaries}. We are going to describe certain moves that one can perform on a walk $W$ to obtain a new walk with the same starting and ending vertices. The reader familiar with the theory of Coxeter groups should recognize these operations as essentially Coxeter moves in the affine symmetric group, whose Coxeter graph is $\Cycle_n$ (though we will not need this formalism). If there are two consecutive identical edges in $W$, then we can simply delete them; we call this operation a \dfn{square deletion} (because it corresponds to the fact that the square of a simple reflection in the affine symmetric group is the identity). The opposite of a square deletion is a \dfn{square insertion}, which consists of inserting two consecutive identical edges into the walk so that the resulting sequence of edges is a walk. If there are two consecutive edges in $W$ that do not share a vertex, then we can simply swap those two edges; we call this operation a \dfn{commutation move}. Finally, if $W$ contains three consecutive edges $e_1,e_2,e_3$, then we can replace them with the edges $e_3,e_2,e_1$; we call this operation a \dfn{Yang--Baxter move}. Collectively, we refer to square deletions, square insertions, commutation moves, and Yang--Baxter moves as \dfn{Coxeter moves}. 

One should think of a commutation move as replacing two edges in a $4$-cycle in $\FS(\Cycle_n,Y)$ with the other two edges in the same $4$-cycle. Similarly, one should think of a Yang--Baxter move as replacing three edges in a $6$-cycle in $\FS(\Cycle_n,Y)$ with the other three edges in the same $6$-cycle. 

We will often specify a walk in $\FS(\Cycle_n,Y)$ by listing its label sequence. It is helpful to note that a Yang--Baxter move acts on the label sequence of a walk by simply reversing the order of three consecutive edge labels of the form $ab,ac,bc$. 

We define an \dfn{anchored walk} to be a walk in $\FS(\Cycle_n,Y)$ whose label sequence is of the form $a_1b_1,a_2b_2,\ldots,a_kb_k,a_1b_1$; the first and last edge labels in this walk (which are identical) are called the \dfn{anchors}. We say that such an anchored walk is \dfn{repetition-free} if the labels $a_1b_1,a_2,b_2,\ldots,a_kb_k$ are distinct. An anchored walk is \dfn{trivial} if its label sequence has only two (necessarily identical) labels. We can perform Coxeter moves on anchored walks in the same way that we performed them on ordinary walks; however, for anchored walks, we preserve the anchors and thus possibly shorten the walk. For example, suppose an anchored walk starts at a vertex $\sigma$ and uses edges with labels $ab,ac,bc,ab$. The anchors are the first and last labels $ab$. We can perform a Yang--Baxter move on the first three edge labels to transform this walk into $bc,ac,ab,ab$. However, as an anchored walk, we would cut out the edge labels that are no longer between the two $ab$ anchors. Thus, the new anchored walk would have label sequence $ab,ab$, so it would be trivial. 

Let us say an anchored walk in $\FS(\Cycle_n,Y)$ is \dfn{complete} if its label sequence is of the form $ab, au_1, \dots au_k, bu_{k+1}, \dots, bu_{n-2}, ab$, where $\{a,b,u_1,\ldots,u_{n-2}\}=[n]=V(\Cycle_n)$. Note that the existence of such a complete walk forces $\{a, b\}$ to be a dominating set of $Y$. We say an anchored walks $W$ \dfn{reduces} to an anchored walk $W'$ if there is a sequence of Coxeter moves that transforms the label sequence of $W$ into that of $W'$ such that the number of square insertions that insert two copies of the label $xy$ is at most the number of square deletions that delete two copies of the label $xy$. 

\begin{theorem} \label{thm:isocycles}
Let $Y$ be an $n$-vertex graph with domination number at least $2$, and let $W$ be a repetition-free anchored walk in $\FS(\Cycle_n,Y)$ with anchors $ab$. If there exists a vertex $z\in V(Y)$ such that the label sequence of $W$ either contains both $az$ and $bz$ or contains neither $az$ nor $bz$, then $W$ reduces to a trivial anchored walk. Otherwise, $W$ reduces to a complete anchored walk.
\end{theorem}

In order to prove Theorem~\ref{thm:isocycles}, it will be convenient to consider some particular subsequences of an anchored walk. Let $Y$ be a graph with vertex set $[n]$. Suppose $W$ is a repetition-free anchored walk in $\FS(\Cycle_n,Y)$ whose anchors are the label $ab$, where $a<b$. We define the \dfn{strong essential prefix} of $W$ to be the longest initial subsequence of the label sequence of $W$ of the form $ab,ax_1,\ldots,ax_k,by_1,\ldots,by_\ell$ such that $a,b,x_1,\ldots,x_k,y_1,\ldots,y_\ell$ are all distinct. The \dfn{strong essential suffix} of $W$ is the part of the label sequence of $W$ not in the strong essential prefix. Similarly, we define the \dfn{weak essential prefix} of $W$ to be the longest initial subsequence of the label sequence of $W$ of the form $ab,ax_1,\ldots,ax_k,by_1,\ldots,by_\ell,c_1d_1,\ldots,c_rd_r$ such that $a,b,x_1,,\ldots,x_k,y_1,\ldots,y_\ell$ are all distinct and $c_1,\ldots,c_r,d_1,\ldots,d_r\in\{x_1,\ldots,x_k,y_1,\ldots,y_\ell\}$. The \dfn{weak essential suffix} is the part of the label sequence of $W$ not in the weak essential prefix. Note that the strong essential prefix is a subsequence of the weak essential prefix and that the weak essential suffix is a subsequence of the strong essential suffix.  

\begin{lemma} \label{lemma1}
Let $Y$ be a graph with vertex set $[n]$. Let $W$ be a repetition-free anchored walk in $\FS(\Cycle_n,Y)$ with anchors $ab$, where $a<b$. If there exists $z\in V(Y)\setminus\{a,b\}$ such that neither $az$ nor $bz$ appear in the label sequence of $W$, then $W$ reduces to a trivial anchored walk. 
\end{lemma}
\begin{proof}
We proceed by induction on the length of $W$. By definition, $W$ is a sequence $\sigma_0,\ldots,\sigma_p$ of vertices in $\FS(\Cycle_n,Y)$. One should think of each vertex $\sigma_i$ as an arrangement of the vertices of $Y$ on the vertices of $\Cycle_n$. In the starting vertex $\sigma_0$, the vertices $a$ and $b$ are sitting on adjacent vertices of $\Cycle_n$, and they swap places when we move to $\sigma_1$. At the end of the walk, we reach $\sigma_p$ from $\sigma_{p-1}$ by swapping $a$ and $b$ again. We have assumed that the label sequence does not contain $az$ or $bz$. This readily implies that if the label sequence of $W$ contains a label of the form $ax$ with $x\neq b$, then it must also contain $bx$. Similarly, if the label sequence contains a label $by$ such that $y\neq a$, then it also contains $ay$. This shows that it suffices to prove that $W$ reduces to an anchored walk whose label sequence does not contain $az$ or $bz$ and whose strong essential suffix is just the single label $ab$ since such an anchored walk must be trivial. Let $m$ be the length of the strong essential suffix of $W$; we may assume $m\geq 2$ since otherwise we are done. By induction on $m$, it suffices to show that $W$ reduces to an anchored walk $W'$ whose label sequence does not contain $az$ or $bz$ and whose strong essential suffix has length at most $m-1$. 

Let $ab, ax_1,...,ax_k, by_1,...,by_\ell$ be the strong essential prefix of $W$. Let $uv$ be the first label in the strong essential suffix of $W$. In other words, $u$ and $v$ are the two vertices of $Y$ that take part in the friendly swap used to get from $\sigma_{k+\ell+1}$ to $\sigma_{k+\ell+2}$. We know that $uv\neq ab$ because $m\geq 2$. Let $Q$ be the part of the strong essential suffix of $W$ that comes after the label $uv$. We now consider three cases. 

\medskip
\noindent {\bf Case 1.} Suppose $\{u,v\}\cap\{a,b,x_1,\ldots,x_k,y_1,\ldots,y_\ell\}=\emptyset$. In this case, we can repeatedly apply commutation moves so that each move swaps the label $uv$ with the label to its left. After we perform $k+\ell+1$ such moves, the label $uv$ will leave the anchored walk, resulting in a new anchored walk $W'$ whose strong essential suffix is contained in $Q$.

\medskip
\noindent {\bf Case 2.} Suppose $\{u,v\}\cap\{x_1,\ldots,x_k,y_1,\ldots,y_\ell\}=\emptyset$ and $\{u,v\}\cap\{a,b\}\neq\emptyset$. We will assume $u=a$ since the other case is similar. We can perform commutation moves to transform the label sequence of $W$ into $ab,ax_1,\ldots,ax_k,uv,by_1,\ldots,by_\ell,Q$; this is the label sequence of the desired anchored walk $W'$ (note that its strong essential suffix has length at most $m-1$ because its strong essential prefix contains $ab,ax_1,\ldots,ax_k,uv,by_1,\ldots,by_\ell$). 

\medskip
\noindent {\bf Case 3.} Suppose $\{u,v\}\cap\{x_1,\ldots,x_k,y_1,\ldots,y_\ell\}\neq\emptyset$. We will assume that $\{u,v\}\cap\{y_1,\ldots,y_\ell\}\neq\emptyset$; the case when $\{u,v\}\cap\{x_1,\ldots,x_k\}\neq\emptyset$ is similar. Let $t$ be the largest index such that $y_t\in\{u,v\}$. Without loss of generality, we may assume $y_t=v$. The fact that $W$ is repetition-free implies that $u\neq b$. Also, the maximality of $t$ guarantees that $u\neq y_{t+1}$ (when $t<\ell$). Since $\sigma_{k+\ell+2}$ is obtained from $\sigma_{k+\ell+1}$ by performing a friendly swap involving the vertices $u$ and $v=y_t$, we know that $\sigma_{k+\ell+1}^{-1}(u)$ must be one of the two vertices of $\Cycle_n$ adjacent to $\sigma_{k+\ell+1}^{-1}(y_t)$. One of these vertices is either $\sigma_{k+\ell+1}^{-1}(b)$ (if $t=\ell$) or $\sigma_{k+\ell+1}^{-1}(y_{t+1})$ (if $t<\ell$); since $u$ is not $b$ or $t_{t+1}$, it follows that $\sigma_{k+\ell+1}^{-1}(u)$ is the other of these two vertices. 

Suppose first that $t\geq 2$. Then $u=y_{t-1}$. We can perform a sequence of commutation moves in order to move the label $uv$ to the left until it is immediately to the right of $by_t$. We can then perform a Yang--Baxter move to transform the subsequence $by_{t-1},by_t,uv$ into $uv,by_t,by_{t-1}$. Finally, since the sets $\{u,v\}$ and $\{a,b,x_1,\ldots,x_k,y_1,\ldots,y_{t-2}\}$ are disjoint, we can apply further commutation moves in order to move $uv$ to the left until it leaves the anchored walk. This results in the desired anchored walk $W'$. 

Next, suppose $t=1$ and $k=0$. Then $u=a$. We can use commutation moves to transform the label sequence of $W$ into $ab,by_1,uv,by_2,\ldots,by_\ell,Q$. Since $uv=ay_1$, we can use a Yang--Baxter move to turn this sequence into $uv,by_1,ab,by_2,\ldots,by_\ell,Q$. The label sequence of the new anchored walk $W'$ does not include the labels $uv$ and $by_1$, so its strong essential suffix has length at most $m-1$. Furthermore, the label sequence of $W'$ does not contain $az$ or $bz$. 

Finally, suppose $t=1$ and $k\geq 1$. Then $u=x_1$, so $uv=x_1y_1$. Let $Q'$ be the sequence $ax_2,\ldots, ax_k, by_2,\ldots, by_\ell, Q$. Because $ax_1$ is in the label sequence of $W$, it follows from the first paragraph of this proof that $bx_1$ appears in $Q$; hence, $b$ and $x_1$ must be adjacent in $Y$. Furthermore, we know $\sigma_2^{-1}(b)$ is adjacent to $\sigma_2^{-1}(x_1)$.
Then we can use a square insertion to change $ab, ax_1, by_1, x_1y_1, Q'$ into $ab, ax_1, bx_1, bx_1, by_1, x_1y_1, Q'$. Applying Yang--Baxter moves to the subsequences $ab, ax_1, bx_1$ and $bx_1, by_1, x_1y_1$ yields $bx_1, ax_1, ab, x_1y_1, by_1, bx_1, Q'$. Another commutation move then yields $bx_1, ax_1, x_1y_1, ab, by_1, bx_1, Q'$.
This shows that we can apply Coxeter moves to transform $W$ into an anchored walk whose label sequence is $ab, by_1, bx_1,ax_2,\ldots,ax_k,by_2,\ldots,by_\ell,Q$. Performing additional commutation moves yields an anchored walk $W_0$ whose label sequence is $ab,ax_2,\ldots,ax_k,by_1,bx_1,by_2,\ldots,by_\ell,Q$. Now, to meet the definition of what it means for an anchored walk to \emph{reduce} to another, we must show that we can perform additional Coxeter moves to delete a pair of labels $bx_1$. As mentioned before, the label $bx_1$ must appear in $Q$. Let $W_1$ be the anchored walk whose label sequence is the subsequence of $W_0$ starting and ending at the labels $bx_1$. Note that $W_1$ is repetition-free and does not use the labels $az$ or $bz$. By induction on the length of the anchored walk, we find that $W_1$ reduces to a trivial anchored walk. This means that we can perform a sequence of Coxeter moves to the label sequence of $W_0$ in order to move the two occurrences of the label $bx_1$ next to each other. We can then perform a square deletion to remove the two occurrences of $bx_1$. Let $W_2$ be the resulting anchored walk with anchors $ab$ (the same as the anchors of $W_0$). When we reduced $W_1$ to the trivial anchored walk, all labels that we added by square insertions were later deleted by square deletions (by the definition of \emph{reducing}). This implies that $W$ reduces to $W_2$, that $W_2$ is repetition-free, and that $W_2$ does not use the labels $az$ or $bz$. Since the length of $W_2$ is strictly less than that of $W$, we can use induction to see that $W_2$, and hence also $W$, reduces to a trivial anchored walk (whose strong essential suffix has length $1$).  
\end{proof}

\begin{lemma}\label{lemma4.9}
Let $W$ be a repetition-free anchored walk in $\FS(\Cycle_n,Y)$. Let \[ab, ax_1, \ldots, ax_k, by_1, \ldots, by_\ell, c_1d_1, \ldots, c_rd_r\] be the weak essential prefix of $W$, where $c_1, \ldots, c_r, d_1, \ldots, d_r \in \{x_1, \ldots, x_k, y_1, \ldots, y_\ell\}$. If $k \ge 1$ and the weak essential suffix of $W$ begins with $bu$ for some $u \in \{x_1,\ldots,x_k\}$, then $W$ reduces to a trivial anchored walk. Similarly, if $\ell \ge 1$ and the weak essential suffix of $W$ begins with $au$ for some $u \in \{y_1,\ldots,y_\ell\}$, then $W$ reduces to a trivial anchored walk.
\end{lemma}
\begin{proof}
We prove only the first statement since the second is similar. Our strategy is to manipulate $W$ through Coxeter moves and apply Lemma~\ref{lemma1}. Let $W$ be the sequence of vertices $\sigma_0, \ldots, \sigma_p$. First, we wish to apply Coxeter moves to move $au$ left until it appears directly after the first anchor. This is already true if $u = x_1$. If $u = x_i$, where $i \ge 2$, then let $Q$ be the subsequence of the label sequence of $W$ consisting of all labels before $ax_{i-1}$, and let $R$ be the subsequence consisting of all labels after $au$. Because $W$ is repetition-free, the labels $ux_1,\ldots,ux_{i-1}$ must all occur in the label sequence between the labels $au$ and $bu$.
This implies that $\{u, x_j\} \in E(Y)$ for all $j \in [i - 1]$. Thus, to move the label $ax_i$ before $ax_{i-1}$, we can use a square insertion to turn the subsequence $ax_{i-1}, au$ into $x_{i-1}u, x_{i-1}u, ax_{i-1}, au$, followed by a Yang-Baxter move on the last three labels, which changes this into $x_{i-1}u, au, ax_{i-1}, x_{i-1}u.$ The first occurrence of $x_{i-1}u$ is disjoint from all labels in $Q$, so we remove it from $W$ via commutation moves. The sequence now begins with $Q, au, ax_{i-1}, x_{i-1}u$. We repeat this process until $au$ appears directly after $ab$, obtaining a new anchored walk $W'$ with the label sequence $ab, au, ax_1, ux_1, ax_2, ux_2, \ldots, ax_{i-1}, ux_{i-1}, R$. Now, we must delete all edges that we added with square insertions. For $j \in [i-1]$, let $W_j$ be the anchored walk with the two occurrences of $ux_j$ in $W'$ as its anchors (the label $ux_j$ occurs exactly once in $R$ for each $j$). The only repeated labels in $W'$ besides the anchors are $ux_j$ for $j \in [i-1]$, so there are no repeated labels in $W_{i-1}$ besides the anchors. Furthermore, the labels $au$ and $ax_{i-1}$ are outside of $W_{i-1}$. Thus, by Lemma \ref{lemma1}, $W_{i-1}$ reduces to a trivial anchored walk, and we can remove its anchors with a square deletion. Now that the labels $ux_{i-1}$ have been deleted, the same argument applies to the anchored walk $W_{i-2}$, and so on; in this way, for each $j \in [i-1]$, we can use a square deletion to remove a pair of edge labels $ux_j, ux_j$. Thus, $W$ reduces to a new anchored walk $W''$ that begins with $ab, au$ and contains $bu$.

Having $au$ directly after $ab$ allows us to use another square insertion to change the subsequence $ab, au$ into $bu, bu, ab, au$. Applying a Yang-Baxter move transforms this into $bu, au, ab, bu$. Now, because $W'$ initially contained the label $bu$, the new anchored walk now contains two instances of $bu$; we have created an anchored walk $B$ within $W''$ with anchors $bu$. Both $au$ and $ab$ appear outside of $B$, so by Lemma \ref{lemma1}, $B$ reduces to a trivial anchored walk, and its anchors can be removed with a square deletion. The labels $bu$ and $au$ do not appear in the resulting anchored walk. Thus, $W''$ reduces to a trivial anchored walk, so the initial walk $W$ also reduces to a trivial anchored walk.
\end{proof}

Now, we are ready to prove the main theorem. 

\begin{proof}[Proof of Theorem \ref{thm:isocycles}]
If there is some $z$ such that neither $az$ nor $bz$ appear in $W$, then $W$ reduces to a trivial anchored walk by Lemma~\ref{lemma1}. For all other $W$, we use induction show that $W$ reduces to an anchored walk that either satisfies the hypothesis of Lemma~\ref{lemma4.9} or has only the ending anchor $ab$ as its weak essential suffix. Let $ab, ax_1, \ldots, ax_k, by_1, \ldots, by_\ell, c_1d_1, \ldots, c_rd_r$ be the weak essential prefix of $W$, where $c_1, \ldots, c_r, d_1, \ldots, d_r \in \{x_1, \ldots, x_k, y_1, \ldots, y_\ell\}$. Let $m$ be the length of the weak essential suffix of $W$. If $m=1$, then we are done. If $m \ge 2$, we consider the first label $uv$ of the weak essential suffix. If $k \ge 1$ and this label is $bu$ for some $u \in \{x_1,\ldots,x_k\}$ or $\ell \ge 1$ and this label is $au$ for some $u \in \{y_1,\ldots,y_\ell\}$, then we are done; $W$ reduces to a trivial anchored walk by Lemma~\ref{lemma4.9}. Otherwise, we consider three cases in a similar way as in the proof of Lemma~\ref{lemma1}.

\medskip
\noindent {\bf Case 1.} Suppose $\{u,v\}\cap\{a,b,x_1,\ldots,x_k,y_1,\ldots,y_\ell\}=\emptyset$. We can apply commutation moves to remove $uv$ from the anchored walk in a similar way as in the proof of Lemma~\ref{lemma1}.

\medskip
\noindent {\bf Case 2.} Suppose $\{u,v\}\cap\{x_1,\ldots,x_k,y_1,\ldots,y_\ell\}=\emptyset$ and $\{u, v\}\cap\{a, b\} \neq\emptyset$. We can again handle this case in a similar way as in Lemma~\ref{lemma1} by performing commutation moves to move $uv$ directly after $ax_k$.

\medskip
\noindent {\bf Case 3.} Suppose $\{u,v\}\cap\{x_1,\ldots,x_k,y_1,\ldots,y_\ell\} \neq \emptyset$. Using the same logic as in the proof of Lemma~\ref{lemma1}, we see that the only possibilities for $uv$ are $ay_1$ if $k=0$ and $\ell \ge 1$, $bx_1$ if $k \ge 1$ and $\ell = 0$, $x_tx_{t-1}$ for some $t \ge 2$, $y_ty_{t-1}$ for some $t \ge 2$, or $x_1y_1$. The cases where $uv$ = $ay_1$ or $bx_1$ are both handled by Lemma~\ref{lemma4.9}. By the definition of a weak essential prefix, $uv$ cannot be $x_tx_{t-1}$, $y_ty_{t-1}$, or $x_1y_1$.

If, throughout the inductive process, $W$ never satisfies the condition in Lemma~\ref{lemma4.9}, then the weak essential suffix of $W$ is just the ending anchor $ab$. We can then apply commutation moves to move each label $c_id_i$ so that they appear after $ab$. Now, the strong essential suffix of the resulting anchored walk $W'$ is $ab$, so $W'$ must be trivial or complete. Furthermore, we know that for each $x \not\in\{ a, b\}$, either $ax$ or $bx$ (not both) must occur in $W'$, which forces $W'$ to be a complete anchored walk.
\end{proof}

Now, we can apply Theorem~\ref{thm:isocycles} to obtain results for when the domination number of $Y$ is at least 3. Recall that a \dfn{geodesic path} in a graph $G$ is a path that has minimum length among all paths with the same endpoints. 

\begin{proposition} \label{prop1.6}
Let $Y$ be an $n$-vertex graph with domination number at least $3$. Let $W$ be a geodesic path in $\FS(\Cycle_n,Y)$. Then no two edges used in $W$ have the same label. 
\end{proposition}
\begin{proof}
Suppose instead that there is a label that appears twice in the label sequence of $W$. Then we can find a subsequence of the label sequence of $W$ that is the label sequence of a repetition-free anchored walk $W'$. Let $ab$ be the anchor label of $W'$. Since $\{a,b\}$ is not a dominating set of $Y$, there must be a vertex $z$ of $Y$ that is not adjacent to $a$ or $b$ in $Y$. Then the labels $az$ and $bz$ cannot appear in the label sequence of $W'$, so it follows from Theorem~\ref{thm:isocycles} that $W'$ reduces to a trivial anchored walk. Once we perform a sequence of Coxeter moves to reduce $W'$ to a trivial anchored walk, we can delete the two consecutive labels $ab$. We can think of performing all of the Coxeter moves on the walk $W$ (which is not necessarily an anchored walk). These moves transform $W$ into a new walk with the same endpoints as $W$ that uses fewer edges than $W$, contradicting the assumption that $W$ is a geodesic path.  
\end{proof}

\begin{proposition}\label{prop1.5}
Suppose $Y$ is an $n$-vertex graph with domination number at least $3$, and let $C$ be a cycle in $\FS(\Cycle_n,Y)$. For every edge $e$ in $C$, there exists an edge $e'\neq e$ in $C$ such that $\psi(e)=\psi(e')$. 
\end{proposition}
\begin{proof}
Let $e=\{\tau,\tau'\}$, and let $\psi(e)=\{a,b\}\in E(Y)$. Let $W$ be the unique walk in $\FS(\Cycle_n,Y)$ that starts at $\tau$, ends at $\tau'$, and uses each edge in $C$ other than $e$ exactly once. We want to show that one of the edges in $W$ has edge label $\{a,b\}$. It is convenient to once again view the vertices of $Y$ as people who are walking on $\Cycle_n$. The arrangement $\tau'$ is obtained from $\tau$ by swapping the people $a$ and $b$, who are standing on adjacent vertices of $\Cycle_n$. Without loss of generality, say $a=\tau(1)=\tau'(2)$ and $b=\tau(2)=\tau'(1)$. As we traverse the walk $W$, we can keep track of the positions of $a$ and $b$ on the cycle. It is straightforward to see that if none of the edges in this walk have label $\{a,b\}$, then each person in $V(Y)\setminus\{a,b\}$ must have swapped with either $a$ or $b$. However, since $\{a,b\}\in E(Y)$, this would imply that $\{a, b\}$ is a dominating set of $Y$, contradicting our assumption that $Y$ has domination number at least $3$.
\end{proof}

Every friends-and-strangers graph is bipartite since its vertices can be viewed as permutations in a symmetric group and performing a friendly swap corresponds to multiplying by a transposition (which changes the sign of the permutation). This tells us that every cycle $C$ in a friends-and-strangers graph has an even number of edges, so it makes sense to talk about pairs of opposite edges and pairs of opposite vertices of $C$. An \dfn{isometric cycle}
of a graph $G$ is a subgraph $H$ of $G$ that is a cycle and has the property that for all vertices $u$ and $v$ of $H$, the distance between $u$ and $v$ in $H$ is the same as the distance between $u$ and $v$ in $G$. The following proposition is an analogue of one that initially appears in \cite{naatz} for $\FS(\Path_n,Y)$.

\begin{proposition}\label{prop:opposite}
Let $Y$ be an $n$-vertex graph with domination number at least $3$. Let $C$ be an isometric cycle in $\FS(\Cycle_n,Y)$. Two edges in $C$ have the same edge label if and only if they are opposite edges of $C$.
\end{proposition}
\begin{proof}
Let $2m$ be the number of edges in $C$. Let $e$ and $f$ be opposite edges in $C$. Proposition~\ref{prop1.5} tells us that there is some edge $e'\neq e$ in $C$ with the same edge label as $e$. Let $W_1,\ldots,W_m$ be the paths within $C$ that have $m$ edges and that use the edge $e$ (note that there are exactly $m$ such paths). Because $C$ is an isometric cycle of $\FS(\Cycle_n,Y)$, each path $W_i$ is a geodesic path in $\FS(\Cycle_n,Y)$. It follows from Proposition~\ref{prop1.6} that $e'$ is not an edge in any of the paths $W_1,\ldots,W_m$. However, the only edge in $C$ that is not in any of the paths $W_1,\ldots,W_m$ is $f$, so $e'=f$. This shows that $e$ and $f$ have the same edge label and that no other edges in $C$ have the same edge label as $e$. 
\end{proof}

\begin{lemma}\label{lemma1.8}
Let $Y$ be an $n$-vertex graph with domination number at least $3$. Let $C$ be an cycle of length at least 8 in $\FS(\Cycle_n,Y)$. Either there are at least two consecutive edges in $C$ that are contained in a common 4-cycle, or there are at least three consecutive edges in $C$ that are contained in a common 6-cycle. Moreover, if $Y$ is triangle-free, then there are at least two consecutive edges in $C$ that are contained in a common 4-cycle. 
\end{lemma}
\begin{proof}
Our proof follows that of \cite[Lemma~4.8]{naatz}. If there are two consecutive edges in $C$ that have disjoint edge labels, then they must belong to a $4$-cycle. Now assume that the edge labels of any two consecutive edges in $C$ have exactly one vertex of $Y$ in common. We will show that $Y$ contains a triangle and that there are at least three consecutive edges in $C$ that are contained in a common 6-cycle. 

Let $e_1,\ldots,e_{k}$ be the edges in $C$, listed in clockwise order. Let $a\in V(Y)$ be the vertex that appears in the edge labels $\psi(e_1)$ and $\psi(e_2)$. Because $\{a\}$ is not a dominating set of $Y$, there must be some $1\leq i\leq k$ such that $a\not\in\psi(e_i)$. (Indeed, if we imagine the vertices of $Y$ sitting on those of $\Cycle_n$, then having $a\in\psi(e_i)$ for all $1\leq i\leq k$ would mean that the vertex $a$ slides all the way around $\Cycle_n$, participating in friendly swaps with all other vertices of $Y$.) Let $r$ be the smallest positive integer such that $a\not\in\psi(e_r)$. By definition of $a$, we have $r\geq 3$. Let $\psi(e_{r-2})=ab$ and $\psi(e_{r-1})=ac$. Let $\sigma$ be the vertex of $\FS(\Cycle_n,Y)$ in both the edges $e_{r-1}$ and $e_r$. Then the two vertices of $\Cycle_n$ adjacent to $\sigma^{-1}(c)$ are $\sigma^{-1}(a)$ and $\sigma^{-1}(b)$. This means that we must have $\psi(e_r)=bc$. Thus, the vertices $a,b,c$ form a triangle in $Y$. Consider the walk in $\FS(\Cycle_n,Y)$ that uses the edges $e_{r-2},e_{r-1},e_r$. We can perform a Yang--Baxter move on this walk to obtain a new walk with three different edges $e_{r-2}',e_{r-1}',e_r'$. The edges $e_{r-2},e_{r-1},e_{r},e_{r}',e_{r-1}',e_{r-2}'$ form a $6$-cycle in $\FS(\Cycle_n,Y)$. 
\end{proof}

Finally, we can prove Theorem~\ref{thm:isometric} by piecing together the results in this section using the same idea as in \cite{naatz}.

\begin{proof}[Proof of Theorem~\ref{thm:isometric}]
Suppose $Y$ has domination number at least $3$. We will prove the theorem when $Y$ is not triangle-free; the proof when $Y$ is triangle-free is completely analogous (we simply invoke the second statement in Lemma~\ref{lemma1.8} instead of the first statement). Let $C$ be a cycle in $\FS(\Cycle_n,Y)$. We will show that $C$ can be written as a symmetric difference of $4$-cycles and $6$-cycles. Since friends-and-strangers graphs are bipartite, $C$ has an even number of edges, say $2m$. If $m\in\{2,3\}$, then we are done. Therefore, we may assume $m\geq 4$ and proceed by induction on $m$. By induction, we just need to show that $C$ can be written as a symmetric difference of cycles of length strictly smaller than $2m$.

It is straightforward to see that a non-isometric cycle in a graph can be written as a symmetric difference of strictly smaller cycles. Hence, we may assume $C$ is isometric. Let $e_1,\ldots,e_{2m}$ be the edges of $C$, listed in clockwise order. If there are two consecutive edges of $C$ contained in a common $4$-cycle, then let $k=2$; otherwise, let $k=3$. According to Lemma~\ref{lemma1.8}, there are at least $k$ consecutive edges in $C$ that belong to a common $2k$-cycle; since $C$ is isometric, there must be exactly $k$ edges in $C$ in this $2k$-cycle. Without loss of generality, say these edges are $e_1,\ldots,e_k$, and let $e_1',\ldots,e_k'$ be the other edges of the common $2k$-cycle $D$. Assume that these edges are named so that $e_1,\ldots,e_k,e_k',\ldots,e_1'$ is the cyclic order that these edges appear around $D$. Our choice of $k$ guarantees that $D$ is an isometric cycle of $\FS(\Cycle_n,Y)$. Therefore, Proposition~\ref{prop:opposite} tells us that $\psi(e_i)=\psi(e_{k+1-i}')$ for all $1\leq i\leq k$. Since $C$ has length at least $8$, we can use Proposition~\ref{prop:opposite} again to see that $e_1',\ldots,e_k'$ are not edges in $C$. Let $C'$ be the cycle with edges $e_1',\ldots,e_k',e_{k+1},\ldots,e_{2m}$. Then $C=C'\triangle D$, so it suffices to show that $C'$ can be written as a symmetric difference of cycles of length strictly smaller than $2m$. Since $C'$ has $2m$ edges, we just need to prove that $C'$ is not isometric. However, this follows from Proposition~\ref{prop:opposite} because $\psi(e_k')=\psi(e_1)=\psi(e_{m+1})$ and the edges $e_k'$ and $e_{m+1}$ are not opposite edges of $C'$.
\end{proof}

\section*{Acknowledgements}

We are grateful to the MIT PRIMES organizers for creating such a rare and amazing math research opportunity.
Colin Defant was supported by the National Science Foundation under Award No.\ DGE--1656466 and Award No.\ 2201907, by a Fannie and John Hertz Foundation Fellowship, and by a Benjamin Peirce Fellowship at Harvard University.

\end{document}